\documentclass[11pt, reqno]{amsart}

\author{Sanjay Mehrotra \and D{\'a}vid Papp}
\title[Generating quadrature formulas]{Generating nested quadrature formulas for\\ general weight functions with known moments}
\date{March 2012; minor update in April 2016.}

\usepackage{amsthm,amsmath,amssymb}
\DeclareMathAlphabet{\mathpzc}{OT1}{pzc}{m}{it}
\usepackage[margin=1.3in]{geometry}
\usepackage{setspace}
\usepackage[colorlinks=true, citecolor=black, linkcolor=black, urlcolor=black, setpagesize=false, plainpages=false, pdfpagelabels]{hyperref}
\usepackage{booktabs}

\newcommand{\real}{\mathbb{R}}

\allowdisplaybreaks[1]

\newtheoremstyle{myexamples}
{3pt}
{3pt}
{}
{}
{\bfseries}
{.}
{.5em}
{{\thmname{#1}\thmnumber{ #2}\thmnote{ (#3)}}}

\theoremstyle{plain}
\newtheorem{theorem}{Theorem}

\newtheorem{proposition}[theorem]{Proposition}

\theoremstyle{definition}

\theoremstyle{myexamples}


\ifpdf\usepackage[final,expansion=true,protrusion=true]{microtype}\fi 

\begin{document}

\maketitle

\begin{abstract}We revisit the problem of extending quadrature formulas for general weight functions, and provide a generalization of Patterson's method for the constant weight function. The method can be used to compute a nested sequence of quadrature formulas for integration with respect to any continuous probability measure on the real line with finite moments. The advantages of the method include that it works directly with the moments of the underlying distribution, and that for distributions with rational moments the existence of the formulas can be verified by exact rational arithmetic.
\end{abstract}

\section{Introduction}

A \emph{quadrature formula} for integration with respect to the \emph{weight function} $\rho\colon\Omega\mapsto\real$ takes the form
\begin{equation}\label{eq:quadrature}
\int_{\Omega}f(t)\rho(t)d{t}\approx \sum_{k=1}^{K}w_k f(t_k).
\end{equation}
The weight function $\rho$ is a strictly positive measurable function that is the probability density function of a continuous random variable with finite moments. The \emph{weights} $w_k$ and \emph{nodes} $t_k$, $k=1,\dots,K$, are chosen so as to maximize the highest degree $d$ for which the approximation \eqref{eq:quadrature} is exact for every polynomial $f$ of degree up to $d$. This degree is called the \emph{degree of polynomial exactness} (sometimes the degree of precision) of the formula. A \emph{quadrature rule} is a sequence of quadrature formulas with an increasing number of nodes and an increasing degree of exactness.

For many applications it is desirable that a quadrature rule be \emph{nested}, that is, that the node set of each formula is a subset of the node set of its successors. To obtain such a sequence, we start with a quadrature formula with $K(1)$ nodes, and extend it to a new formula with higher degree of exactness by adding $K(2)-K(1)$ additional nodes. The weights of the existing nodes may change. Repeated application of such an algorithm shall give rise to a nested sequence of quadrature formulas of increasing degree of polynomial exactness.

In 1964, Kronrod presented a method to extend the well-known Gauss--Legendre formulas \cite{Kronrod1965}. His construction adds, for any $K$, $K+1$ nodes to the $K$-node Gauss--Legendre formula (which has degree of exactness $2K-1$ for the constant weight function $\rho\equiv1/2$ on $\Omega=[-1,1]$), extending it to a formula with $2K+1$ nodes and a degree of exactness at least $3K+1$. He also showed that this is the best possible extension (in terms of the achieved degree of exactness), but he did not consider longer sequences of nested formulas. Patterson \cite{Patterson1968} showed that Kronrod's method can be used to obtain nested rules by iteratively extending the extended formulas. He also considered the constant weight function over the interval $[-1,1]$, the resulting quadrature rule is now known as the Gauss--Kronrod--Patterson (or GKP) rule.

It is possible to generalize Patterson's method to obtain nested quadrature rules for other, non-uniform, continuous distributions with finite moments. To the best of our knowledge, this is very little known or used in the literature. One known example is the quadrature rule proposed by Genz and Keister \cite{GenzKeister1996} for integration with the Gaussian weight function. Their approach is a direct generalization of Patterson's 1964 algorithm. Patterson's 1989 paper \cite{Patterson1989} also considers general weight functions, but that algorithm is not a direct generalization of the first.

Patterson's 1989 method requires that the distribution is given by the three-term linear recurrence relation satisfied by orthogonal polynomial bases with respect to the probability density function of the distribution. Obtaining this recurrence may itself be a difficult task \cite{GolubWelsch1969}, and the coefficients of the recurrence is not as commonly available information for known distributions as moments are. In this note we formally propose and give the details of an algorithm that generates nested sequences of quadrature formulas for general continuous distributions with finite moments. The algorithm circumvents the use of the three-term linear recurrence, and works with the moments of the underlying distribution directly. This yields a streamlined version of Patterson's algorithm, which can also be easily implemented. A Mathematica implementation is also provided.


\section{Extending quadrature formulas using known moments}

Our quadrature formula extension algorithm relies on the following two results. The first one is an immediate generalization of \cite[Theorem 1]{Kronrod1965}; its proof is omitted.

\begin{proposition}\label{lem:GKP-1}
For every probability density function $\rho$ and every set $\{t_1,\dots,t_{K}\}\subseteq \real$ of $K$ nodes there exists unique weights $w_1,\dots,w_{K}$ such that the quadrature formula for integration with respect to $\rho$ has a polynomial degree of exactness at least $K-1$. These weights are the unique solution of the linear (square) system of equations
\begin{equation}\label{eq:Patterson1}
\sum_{i=1}^{K} w_i t_i^k = \int_\Omega t^k \rho(t)dt, \quad k=0,\dots,K-1.
\end{equation}
\end{proposition}

\begin{theorem}\label{thm:GKP-2}
Let $\rho$ be the probability density function of a distribution supported on $\Omega\subseteq\real$ with finite moments. Let $F$ be a univariate polynomial of degree $n$ with $n$ distinct real roots, and suppose that there exists a polynomial $G$ of degree $p$ satisfying
\begin{equation}\label{eq:GKP-1}
\int_\Omega F(t)G(t)t^i\rho(t)dt = 0,\quad i=0,\dots,p-1.
\end{equation}
Assume further that the roots of $G$ are all real and of multiplicity one, and distinct from those of $F$. Then there exists a quadrature formula supported on the roots of $FG$, whose degree of polynomial exactness is at least $n+2p-1$.
\end{theorem}

\begin{proof}
The polynomial $FG$ has degree $n+p$, therefore every polynomial $H$ of degree $n+2p-1$ can be written as $H = QFG + R$ for some polynomials $Q$ of degree $p-1$ and $R$ of degree $n+p-1$. Consider now the quadrature formula supported on the roots of $FG$ with degree of polynomial exactness at least $n+p-1$, whose existence is established in the previous Lemma (with $n+p$ in place of $K$). Denoting by $q(H)$ the value $\sum_{i=1}^{n+p}w_iH(t_i)$ of this quadrature formula for the polynomial $H$ we have that $q(R) = \int_\Omega R(t)\rho(t)dt$ because $R$ has degree $n+p-1$; $q(QFG)=0$ because $q$ is supported on the roots of $FG$; and $\int_\Omega Q(t)F(t)G(t)\rho(t)dt=0$ owing to \eqref{eq:GKP-1} and the fact that $Q$ is of degree $p-1$. Hence,
\begin{align*}
\int_\Omega H(t)\rho(t)dt &= \int_\Omega R(t)\rho(t)dt + \int_\Omega Q(t)F(t)G(t)\rho(t)dt =\\
&= q(R) + 0 = q(R) + q(QFG) = q(H).
\end{align*}
Therefore, the formula gives the correct value of $\int_\Omega H(t)\rho(t)dt$ for every $H$ of degree $n+2p-1$.
\end{proof}

These assertions suggest the following algorithm:
\begin{enumerate}
	\item Consider an arbitrary quadrature formula on $n$ nodes, and the polynomial $F$ whose roots (of multiplicity one) are the nodes of the formula.
	\item Choose an integer $p \geq 1$, and find a degree $p$ polynomial $G$ that is not identically zero and that solves the system of equations \eqref{eq:GKP-1}. Eq.~\eqref{eq:GKP-1} is a homogeneous system of linear equations with one fewer variables than equations, but we can assume that the leading coefficient of $G$ is $1$, and turn \eqref{eq:GKP-1} into an inhomogeneous square system of linear equations.
	\item Determine the roots of $G$; these are the new nodes of the extended formula. Now find the weights of the formula by solving \eqref{eq:Patterson1}.
\end{enumerate}

By Theorem \ref{thm:GKP-2}, the resulting formula adds new nodes to the initial formula, potentially increasing its degree of polynomial exactness. If $p > n/2$, the second formula necessarily has a higher degree of exactness, since an $n$-node formula cannot have a degree of exactness higher than $2n-1$. Repeated application of this algorithm may yield a nested sequence of quadrature formulas of increasing degree of exactness.

As with Patterson's 1964 algorithm, this algorithm may also fail to yield a new formula if either the linear system \eqref{eq:GKP-1} has no nonzero solution, or $G$ has complex roots or real roots of multiplicity higher than one, or if $F$ and $G$ have common roots. If the algorithm fails for a given $p$, different values of $p$ may be tried sequentially. Finally, different initial formulas may give rise to different sequences.

\section{Further remarks}

The original GKP formulas can be obtained using the above algorithm with $\Omega = [-1,1]$ and $\rho(t) = 1/2$, starting with the trivial $1$-node formula (with the node at $0$, with weight $1$), and taking $p=n+1$ in each iteration. The number of nodes is thus $2^{i+1}-1$ after $i$ iterations. It is not known whether this process can be repeated indefinitely, but formulas up to $511$ nodes have been determined. The Genz--Keister formulas mentioned above can be obtained by successively applying the above algorithm to the normal distribution.

It should be noted that although the computation of the nodes and weights may be numerically challenging (especially since the matrix of the linear system \eqref{eq:GKP-1} is an ill-conditioned Hankel matrix), the \emph{existence} of the solution for a given $p$ can be decided using exact rational arithmetic for every distribution whose moments are rational numbers. In this case, the coefficients of $G$ (if $G$ exists) are rational numbers, and the number of those roots of $G$ that are distinct from the roots of $F$ can be determined without the computation of the roots of either $F$ or $G$ \cite[Chap.~2]{BPR-2003}.

A Mathematica implementation of the algorithm, which runs with exact rational arithmetic or extended precision arithmetic based on the input, is available from the authors, and is given in the Appendix, where a numerical example is also provided.

\appendix

\section{Mathematica code}
A simple Mathematica implementation of the quadrature rule generation algorithm consists of two functions. The first one, \texttt{FormulaExtension[F,p,moments,\{var,a,b\},prec]} returns the polynomial $G$ satisfying the conditions of \mbox{Theorem \ref{thm:GKP-2}}, or $0$ if such a polynomial does not exist.

\begin{verbatim}
FormulaExtension[F_, p_, moments_, {var_, a_, b_}, prec_: \[Infinity]] :=
Module[{G = Sum[g[i] var^i, {i, 0, p}], 
        mts = Take[moments, Exponent[F, var] + 2 p + 1],
        integrals, M, ls, success = False, ep},
   If[prec < \[Infinity], mts = SetPrecision[mts, prec]];
   integrals = Table[CoefficientList[F var^ipj, var].
      Table[mts[[i+1]], {i,0,Exponent[F var^ipj, var]}], {ipj, 0, 2 p}];
   M = HankelMatrix[Take[integrals, (Length[integrals]+1)/2],
      Take[integrals, -(Length[integrals]+1)/2]];
   Quiet[ls = LinearSolve[Join[Drop[M, -1], {Append[Array[0 &, p], 1]}],
      Append[Array[0 &, p], 1]];
      If[Head[ls] =!= LinearSolve,
         ep = ls.var^Range[0, p];
         success = CountRoots[ep,{var,a,b}]==p &&
            Resultant[F,ep,var]!=0 && Discriminant[ep,var]!=0
      ]
   ];
   If[success, ep, 0]
]
\end{verbatim}

The input arguments \texttt{F} and \texttt{p} of \texttt{FormulaExtension} are as in Theorem \ref{thm:GKP-2}; \texttt{moments} is the list of known moments; the list \texttt{\{var,a,b\}} contains the symbol for the variable of $F$ and the the boundaries of the interval $\Omega=[a,b]$; and \texttt{prec} is an optional argument that controls the precision of the numerical calculations. Leaving it at its default value, infinity, the calculations are carried out with the lowest precision of all the inputs. If the coefficients of $F$, the moments, and the endpoints of the interval $[a,b]$ are all given as (symbolic) rational numbers, $G$ all computations are carried out using exact rational arithmetic.

The moments can be obtained, for instance, by symbolic integration using Mathematica. The moments of well-known distributions are also well-known, their precomputed values can be accessed through the function \texttt{Moment}.

The second function needed to compute the formulas is \texttt{NodesAndWeights[F,moments,var,prec]}, which returns the nodes and weights of the formula supported on the roots of \texttt{F} by solving \mbox{Eq.~\eqref{eq:Patterson1}}.

\begin{verbatim}
NodesAndWeights[F_, moments_, var_, prec_: $MachinePrecision] :=
Module[{w, expF = Exponent[F, var], 
        nodes = var /. {ToRules[NRoots[F==0,var,PrecisionGoal->prec]]}},
  Transpose@{
     nodes,
     Array[w, expF] /. First@Solve@Table[
       Sum[w[k] If[nodes[[k]]==0 && i==0, 1, nodes[[k]]^i], {k, expF}] ==
       moments[[i+1]], {i,0, expF-1}]
  }
]
\end{verbatim}

\section{Example}
We illustrate the approach by computing nested quadrature formulas for the Beta($1/2$,$1/2$) distribution. First, we compute the moments of the distribution:
\begin{verbatim}
   moms = Moment[BetaDistribution[1/2, 1/2], #] & /@ Range[0, 100];
\end{verbatim}
Then, starting with the ``empty'' formula (corresponding to the polynomial constant $1$), we find the polynomials corresponding to the successive extensions of formulas:
\begin{verbatim}
   F = 1;
   F = F FormulaExtension[1, 1, moms, {t, 0, 1}]
   F = F FormulaExtension[F, 2, moms, {t, 0, 1}]
   F = F FormulaExtension[F, 4, moms, {t, 0, 1}]
   F = F FormulaExtension[F, 6, moms, {t, 0, 1}]
   F = F FormulaExtension[F, 12, moms, {t, 0, 1}]
\end{verbatim}
After the last step we obtain the polynomial
{\scriptsize
\begin{align*}
& F(t)= \left(t-\frac{1}{2}\right)\left(t^2-t+\frac{1}{16}\right)\left(t^4-2 t^3+\frac{19
   t^2}{16}-\frac{3 t}{16}\right) \left(t^6-3
   t^5+\frac{27 t^4}{8}-\frac{7 t^3}{4}+\frac{105
   t^2}{256}-\frac{9 t}{256}+\frac{1}{2048}\right)\cdot\\
& \cdot \left(t^{12}-6 t^{11}+\frac{63 t^{10}}{4}-\frac{95
   t^9}{4}+\frac{2907 t^8}{128}-\frac{459
   t^7}{32}+\frac{1547 t^6}{256}-\frac{429
   t^5}{256}+\frac{19305 t^4}{65536}-\frac{1001
   t^3}{32768}+\frac{429 t^2}{262144}-\frac{9
   t}{262144}+\frac{1}{8388608}\right).
\end{align*}}%

Finally, we can obtain the formulas by running \texttt{NodesAndWeights[F,moms,t,prec]} after each call of \texttt{FormulaExtension} above, with our choice of precision \texttt{prec}. Table \ref{tbl:beta1p21p2} shows the nodes of the obtained formulas.

The running time of the five calls to \texttt{FormulaExtension} was under $0.1$ seconds on a standard laptop with a 2.83 Ghz CPU running Mathematica 8.0.1. The CPU time required for the computing the weights depends on the precision. With \texttt{prec=100} the weights of the five formulas were obtained in $0.4$ seconds.

\begin{table}
\centering
\begin{tabular}{lc}
\toprule
node & index \\
\midrule
0. & 3 \\
0.0042775693130947944277212365357185643611308627759489 & 5 \\
0.017037086855465856625128400135551316183047580495798 & 4 \\
0.038060233744356621935908405301605856588791687068179& 5 \\
0.066987298107780676618138414623531908264298686547405 & 2 \\
0.10332332985438241771011151924935036168566203947404 & 5 \\
0.14644660940672623779957781894757548035758203115576 & 4 \\
0.19561928549563968029195122855091799774180314401876 & 5 \\
0.25 & 3 \\
0.30865828381745511413577000798480056661932771875719 & 5 \\
0.37059047744873961882555058118797583582546554934003 & 4 \\
0.43473690388997420422579688605225549490312964759413 & 5 \\
0.5 & 1 \\
0.56526309611002579577420311394774450509687035240587 & 5 \\
0.62940952255126038117444941881202416417453445065997 & 4 \\
0.69134171618254488586422999201519943338067228124281 & 5 \\
0.75 & 3 \\
0.80438071450436031970804877144908200225819685598124 & 5 \\
0.85355339059327376220042218105242451964241796884424 & 4 \\
0.89667667014561758228988848075064963831433796052596 & 5 \\
0.93301270189221932338186158537646809173570131345260 & 2 \\
0.96193976625564337806409159469839414341120831293182 & 5 \\
0.98296291314453414337487159986444868381695241950420 & 4 \\
0.99572243068690520557227876346428143563886913722405 & 5 \\
1. & 3 \\
\bottomrule\\
\end{tabular}
\caption{Nodes of a nested quadrature rule for the Beta($1/2$,$1/2$) distribution. The index column shows the smallest $k$ for which the node appears in the $k$th formula. The weights of the formulas are omitted for brevity.}\label{tbl:beta1p21p2}
\end{table}

\bibliographystyle{amsplain}
\bibliography{patterson}

\providecommand{\bysame}{\leavevmode\hbox to3em{\hrulefill}\thinspace}
\providecommand{\MR}{\relax\ifhmode\unskip\space\fi MR }
\providecommand{\MRhref}[2]{%
  \href{http://www.ams.org/mathscinet-getitem?mr=#1}{#2}
}
\providecommand{\href}[2]{#2}
\begin{thebibliography}{1}

\bibitem{BPR-2003}
Saugata Basu, Richard Pollack, and Marie-Franc{\c{o}}ise Roy, \emph{Algorithms
  in real algebraic geometry}, Springer, Berlin/Heidelberg, 2003.

\bibitem{GenzKeister1996}
Alan Genz and B.D. Keister, \emph{Fully symmetric interpolatory rules for
  multiple integrals over infinite regions with {G}aussian weight}, Journal of
  Computational and Applied Mathematics \textbf{71} (1996), no.~2, 299--309.

\bibitem{GolubWelsch1969}
Gene~H. Golub and John~H. Welsch, \emph{Calculation of {G}auss quadrature
  rules}, Mathematics of Computation \textbf{23} (1969), no.~106, 221--230.

\bibitem{Kronrod1965}
Aleksandr~Semenovich Kronrod, \emph{Nodes and weights of quadrature formulas:
  Sixteen-place tables}, Consultants Bureau, New York, {NY}, 1965.

\bibitem{Patterson1968}
T.~N.~L. Patterson, \emph{The optimal addition of points to quadrature
  formulae}, Mathematics of Computation \textbf{22} (1968), no.~104, 847--856.

\bibitem{Patterson1989}
T.~N.~L. Patterson, \emph{An algorithm for generating interpolatory quadrature
  rules of the highest degree of precision with preassigned nodes for general
  weight functions}, ACM Transactions on Mathematical Software \textbf{15}
  (1989), 123--136.

\end{thebibliography}

\end{document}